\documentclass{siamart171218}
\usepackage{epsfig, graphicx}
\usepackage{graphicx}
\graphicspath{ {./figures} } 

\usepackage{latexsym,amsfonts,amsbsy,amssymb}
\usepackage{amsmath,amscd}
\usepackage{cite}

\renewcommand{\vec}[1]{\ensuremath{\boldsymbol{#1}}}

\newcommand{\eps}{{\varepsilon}}
\newcommand{\Clow}{\underline{C}}
\newcommand{\Chigh}{\overline{C}}
\newcommand{\Oh}{\mathcal{O}}

\newcommand{\fix}[1]{{#1}}

\title{A Boundary-Layer Preconditioner for Singularly Perturbed
  Convection Diffusion\thanks{Submitted to the editors DATE.
\funding{The work of S.M. was partially funded by an NSERC Discovery
  Grant.  The authors wish to acknowledge the Irish Centre for High-End
Computing (ICHEC) for the provision of computational facilities and
support. 
}}}
\author{Scott P. MacLachlan\thanks{Department of Mathematics and Statistics, Memorial University of Newfoundland, St. John's, NL A1C 5S7, Canada
    (\email{smaclachlan@mun.ca}).}  \and Niall
  Madden\thanks{\fix{School of Mathematical and Statistical Sciences},
    National University of Ireland, Galway, Ireland (\email{Niall.Madden@NUIGalway.ie})} \and Th\'{a}i Anh Nhan\thanks{Department of Mathematics and Science,
    Holy Names University, 3500 Mountain Blvd.,
    Oakland, CA 94619, USA (\email{nhan@hnu.edu})}}

\newsiamremark{remark}{Remark}
\begin{document}

\maketitle

\begin{abstract}
  Motivated by a wide range of real-world problems whose solutions
exhibit boundary and interior layers, the numerical analysis of
discretizations of singularly perturbed differential equations is an
established sub-discipline within the study of the numerical approximation
of solutions to
differential equations.  Consequently, much is known about how to
accurately and stably discretize such equations on \textit{a priori}
adapted meshes, in order to properly resolve the layer structure
present in their continuum solutions.  However, despite being a key step in the
numerical simulation process, much less is known about the efficient and
accurate solution of the linear systems of equations corresponding to
these discretizations.

In this paper, we discuss problems associated with the application of
direct solvers to these discretizations, and we 
propose a preconditioning strategy
that is tuned to the matrix structure induced by using layer-adapted
meshes for convection-diffusion equations, proving a strong
condition-number bound on the preconditioned system in one spatial
dimension, and a weaker bound in two spatial dimensions.  Numerical
results confirm the efficiency of the resulting preconditioners in one
and two dimensions, with
time-to-solution of less than one second for representative problems
on $1024\times 1024$ meshes and up to $40\times$ speedup over standard
sparse direct solvers.
\end{abstract}
\begin{keywords}
Singularly Perturbed Differential Equations; Stable Finite-Difference Discretization; Preconditioning;
Domain Decomposition; Multigrid Methods
\end{keywords}
\begin{AMS}
65F08, 65N22, 65N55
\end{AMS}

\section{Introduction}

We are interested in the design and implementation of
efficient linear solvers for discretizations of singularly
perturbed problems of the form
\begin{equation}\label{eq:1D}
  -\eps u'' - c(x)u' + r(x)u  = f \text{ on } (0,1),
\end{equation}
and
\begin{equation}\label{eq:2D}
  -\eps \Delta u - \vec c(x,y) \cdot \nabla u + r(x,y)u  = f \text{ on } (0,1)^2,
\end{equation}
subject to homogeneous Dirichlet boundary conditions. Here, $\eps \in
(0,1]$ is referred to as the ``perturbation parameter''; in the
cases of primary interest, $\eps \ll 1$.  For smooth forcing functions,
$f$, the behaviour of the solution, $u$, is known to be different for
the reaction-diffusion case (with $c=0$ or $\vec c = \vec 0$) and the
convection-diffusion case (with $c\neq0$ or $\vec c \neq \vec 0$).  As a
result, different discretizations and solver approaches may be appropriate and
effective in the two cases.  Here, we focus on the
convection-diffusion case; a similar strategy for reaction-diffusion problems
was previously proposed and analysed in
\cite{SMacLachlan_NMadden_2013a}, although we note that, as usual, the
non-symmetric case requires much different techniques than the
symmetric and positive-definite one considered therein.

Equations such as \eqref{eq:1D} and \eqref{eq:2D} and their many
variants are common in mathematical modelling since their solutions
exhibit boundary and/or
interior layers. \fix{For example, Morton lists ten typical problems which
feature equations of these types, including models of water and
atmospheric pollution, electric currents in semi-conductors, 
turbulent transport, and financial derivatives~\cite[Chap 1.]{Morton96}. Finite-difference
methods, and upwind schemes in particular, have long been used for
their discretization~\cite[Chap 10]{LeVeque07}}.

\fix{In the context of singularly perturbed problems}, the challenge
for numerical analysts is, usually, to
design implementable methods that resolve any layers present and
guarantee a certain order of convergence with respect to the mesh size,
independent of $\eps$. 
Methods with such a property 
are referred to as ``parameter robust''. Typically, they are
specialist methods involving highly non-uniform meshes, often combined
with  nonstandard discretizations \cite{LiSt12, JAdler_etal_2014d}.
The crux of the issue, for all singularly perturbed problems, is that
the layers are the regions of greatest interest, but they are located in
very narrow regions, with widths that may be as small as
$\mathcal{O}(\eps)$. Thus, resolution of these layers with uniform
grids would require a mesh resolution is
$\mathcal{O}(\eps)$. This is not possible as $\eps \to 0$.  Standard
adaptive mesh refinement techniques, such as $h$- and $p$-refinement
also fail to be robust, since many levels of refinement are needed to
resolve the layer regions when starting from a uniform mesh.  Thus, a
common approach is to use \textit{a priori} adapted meshes that are
chosen to resolve the (known) layer regions \cite{Linss10}.

A further complication is introduced for convection-dominated
problems such as those above: it is well understood that classical
methods, such as central finite-difference methods, yield highly
oscillatory numerical solutions unless, again,  mesh resolution is 
$\mathcal{O}(\eps)$~\cite{RoSt08,StSt18}.
That is, even if one were only interested in qualitatively accurate
solutions away from the layers, one still needs an infeasibly large number
of degrees of freedom.

The solution to this is to use stabilized discretizations along with layer-adapted
grids. In this setting, there are many different approaches for constructing stable
discretizations. We will focus on the simplest, and, arguably, most
commonly used: upwind finite-difference methods. There are many
proposed layer adapted meshes in the literature (see, e.g.,
\cite{Linss10}). For our exposition,
we focus on the most widely studied: the piecewise
uniform mesh of Shishkin~\cite{MiOR12}. However, the analysis extends
immediately to more general layer-adapted meshes; see
Remark~\ref{rem:not just S-meshes}.

There is a rich mathematical theory underpinning the
parameter robustness of upwind finite-difference methods on layer
adapted meshes. However, almost exclusively, this work ignores the issue of
solving the resulting linear systems. This is an oversight, since it
is known that standard direct solvers are surprisingly
inefficient when applied to these
discretizations~\cite{SMacLachlan_NMadden_2013a,TANhan_etal_2018a}. Furthermore,
the convergence of standard iterative methods deteriorates drastically as $\eps
\to 0$, unless specialized preconditioning is employed.

The numerical solution of the linear systems resulting from the upwind
finite-difference discretization of these problems was first
considered in \cite{Ro96}, where it was shown that the condition
number of the one-dimensional problem discretized on Shishkin meshes
with $N$ points scales like $\Oh(N^2/(\eps\ln^2(N)))$, but that a
diagonal preconditioning can be defined to improve this to
$\Oh(N^2/\ln(N))$. Similarly, the performance of Gauss-Seidel for the
one-dimensional problem was considered in \cite{FaSh98}, while
numerical experiments for the two-dimensional problem with Incomplete
LU preconditioners were performed in \cite{AnHe03}.  More robust
preconditioning strategies have also been considered. In particular,
\cite{MaOR02} consider an overlapping multiplicative Schwarz method
for the one-dimensional problem, with \textit{volumetric overlap}
(meaning the overlap between two subdomains has positive volume),
building on existing work showing that this gives parameter-robust
solution of the corresponding continuous Schwarz method
\cite{TPMathew_1998a}.  They prove convergence of the corresponding
discrete Schwarz algorithm as well, and demonstrate robustness of
their technique both with respect to the singular perturbation
parameter and the mesh size, using exact solves of the resulting
subdomain systems.

More recently, \cite{CEcheverria_etal_2018a}
proposed a multiplicative-Schwarz preconditioner for the
one-dimensional model with \textit{minimal overlap} (meaning that
adjacent subdomains share only a single mesh point), again with exact
subdomain solves.  This results in a rank-one spectral structure for
the iteration matrix that allows a full convergence analysis, but
precludes extension of these results to two-dimensional models.  Robust
solution methods based on multigrid principles have also been
considered, with \cite{GaClLi02} considering the two-dimensional model
using a standard multigrid approach for anisotropic differential
operators, based on full coarsening multigrid with rediscretization
determining coarse-grid operators.  In this approach, the inter-grid
transfer operators must be adapted, in order to account for the
Shishkin mesh structure, and alternating line Gauss-Seidel relaxation
is used to account for the resulting anisotropy.

The present work is distinguished by a number of features, though
most prominently the development of a special boundary-layer
preconditioner, which proposes distinct treatment of the different
regions that are induced when using tensor-product layer-adapted meshes.
This is in the spirit of the preconditioner proposed
in~\cite{SMacLachlan_NMadden_2013a} for a reaction-diffusion problem,
but is entirely different in implementation and analysis. The
complications  are due both to the non-symmetric discretization matrices, and the
fact that,  in two dimensions, the mesh has
corner-layer regions where the mesh spacing in the $x$- and
$y$-directions  may, or may not, be
highly anisotropic.
This work is also distinguished by combining thorough  analysis of
the preconditioner in one dimension, with a detailed development of
an efficient preconditioner for two distinct two-dimensional cases,
with supporting heuristics.

\paragraph{Outline} This article in organized as follows.
We consider a one-dimensional problem in Section~\ref{sec:1D}.
The upwind finite-difference method, and layer-adapted Shishkin
mesh, are described in Sections~\ref{sec:discretization}
and~\ref{sec:1D S-mesh}, respectively. In Section~\ref{sec:1D M}, we
propose an idealized preconditioner for this problem, and provide a
detailed analysis in Section~\ref{sec:1D theory}, culminating in
proven $\eps$-robust bounds on the spectrum of the preconditioned
system. We conclude the study of the one-dimensional problem in
Section~\ref{sec:1D numbers} with a presentation of numerical results
verifying the uniform convergence of the difference scheme, a
discussion on a suitable stopping criterion for GMRES, and iteration
counts for preconditioned GMRES, showing the effectiveness of the
strategy as $\eps \to 0$.

Section~\ref{sec:2D} is devoted to two-dimensional problems. The solutions
to such equations can be very different in their nature; depending on
$\vec{c}$ in \eqref{eq:2D}, they may possess intersecting layers that
are both exponential in nature, or a mixture of
exponential and parabolic in nature. We describe the meshes and
discretizations for both these cases. In Section~\ref{sec:2D M}, we present the construction of
a suitable preconditioner for both cases,  focusing on the (coarse)
interior region and anisotropic edge-layer regions, and examining the
convergence rate when an idealized preconditioner is used in the
corner region.  The practical task of preconditioning the corner
region, with a multigrid approach, is discussed at length in
Section~\ref{sec:multigrid in the corner}. In
Section~\ref{sec:2D numbers}, we present results for two test problems,
showing the robust convergence of the scheme, the scaling of the
solving times with $N$ and $\eps$, the robustness of the iteration
counts with respect to $N$ and $\eps$, and the speedup achieved over a
best-in-class direct solver.  Conclusions and remarks on potential
future work are given in Section~\ref{sec:conclusions}.

\section{The one-dimensional problem}
\label{sec:1D}
Recall the one-dimensional problem \eqref{eq:1D}, subject to the
boundary conditions $u(0)=u(1)=0$. We shall assume that
\begin{equation}
  \label{eq:coef assumptions}
  0 < \Clow \leq c(x) \leq \Chigh \quad \text{ and } \quad
  0 \leq r(x) \quad \text{ for all } x \in [0,1].
\end{equation}

\subsection{Discretization}
\label{sec:discretization}
We first introduce the (for now, arbitrary) mesh
$\Omega^N:=\{0=x_0< x_1< \dots < x_N=1\}$. We use the notation
$c_i := c(x_i)$, $r_i := r(x_i)$, $h_i = x_i - x_{i-1}$, and
$\bar h_i = (h_i + h_{i+1})/2$.

The standard (centred) second-order finite-difference scheme on this mesh has as
its stencil, at mesh point $i$,
\begin{equation}\label{eq:central stencil}
  \left[ -\frac{\eps}{h_i\bar h_i} + \frac{c_i}{2{\bar h_i}},
    \frac{\eps}{\bar h_i}\left(\frac{1}{h_i} +
      \frac{1}{h_{i+1}}\right) + r_i ,
  -\frac{\eps}{h_{i+1}\bar h_i} -\frac{c_i}{2{\bar h_i}} \right].
\end{equation}
Although this scheme is formally second-order (at least on uniform meshes), it is well known (see,
e.g,~\cite[Remark 3.1]{StSt18}) that it
usually  leads to highly oscillatory solutions when applied on a
uniform mesh. Various explanations are possible. Here, we note that the
scheme associated with \eqref{eq:central stencil} cannot yield
oscillatory solutions if the matrix is an M-matrix~\cite{Berman1994}, which will be the case if it is diagonally dominant
and has its only positive entries on the diagonal. However, it is
clear that the subdiagonal entry is \fix{non-positive only when $\eps \geq c_i h_i/2$}.
For arbitrarily small  $\eps$, generating a discretization matrix that is an M-matrix would, thus, require that
$h_i$ be $\mathcal{O}(\eps)$ for each $i$.
Since that is not a reasonable requirement, an upwind finite-difference
method is preferred, even though it is formally of lower order. Its stencil, at mesh point $i$, is
\begin{equation}
  \label{eq:stencil_1d}
  \left[ -\frac{\eps}{h_i\bar h_i} , \frac{\eps}{\bar h_i}\left(\frac{1}{h_i} +
      \frac{1}{h_{i+1}}\right) + \frac{c_i}{h_{i+1}} + r_i ,
    -\frac{\eps}{h_{i+1}\bar h_i} -\frac{c_i}{h_{i+1}} \right].
\end{equation}
Although this scheme is only first-order, its system matrix is an
M-matrix for any $\eps$ and mesh, and solutions cannot 
feature  spurious oscillations on any mesh.

\subsection{Boundary layer fitted meshes}
\label{sec:1D S-mesh}

One can apply the scheme in \eqref{eq:stencil_1d} on a uniform mesh to solve
\eqref{eq:1D} numerically, and, at mesh points, the computed solution
will be quantitatively and qualitatively reasonable. However, layers
present in the true solution to \eqref{eq:1D} will not be resolved
and, consequently, the global error (e.g., between any polynomial
interpolant of the numerical solution and the true solution of \eqref{eq:1D})
will be $\mathcal{O}(1)$. The  simplest remedy for this is to use a
specially constructed \emph{layer-adapted mesh}~\cite{Linss10}. There
are numerous varieties of these, but they share the same basic
construction: in the layer region (whose location is determined
\emph{a priori} using qualitative analysis techniques), the mesh is
very fine, while elsewhere it is coarse and uniform.

We focus here on the much studied \emph{Shishkin mesh}~\cite{MiOR12}.
For \eqref{eq:1D},
subject to the assumptions in \eqref{eq:coef assumptions}, we define a
\emph{mesh transition point}
\begin{equation}
  \label{eq:tau 1D}    
  \tau = \min\left\{\frac{1}{2},\frac{2\eps}{\Clow}\ln(N)\right\}.
\end{equation}
Then the mesh is constructed by forming uniform grids with $N/2$
intervals on each of the subdomains $[0, \tau]$ and $[\tau,1]$; see
\autoref{fig:1D mesh}.
Then, if $u$ is the solution of~\eqref{eq:1D} evaluated at the mesh points, and $U^E$ denotes
the solution computed on this mesh using the upwind
scheme of \eqref{eq:stencil_1d}, it can be proven that
\begin{equation}
  \label{eq:error bound 1D Shishkin}
  \|u-U^E\|_\infty \leq C N^{-1} \ln N,
\end{equation}
where the constant $C$ is independent of $N$ and $\eps$; see, e.g.,
\cite[Thm. 3.39]{StSt18}.
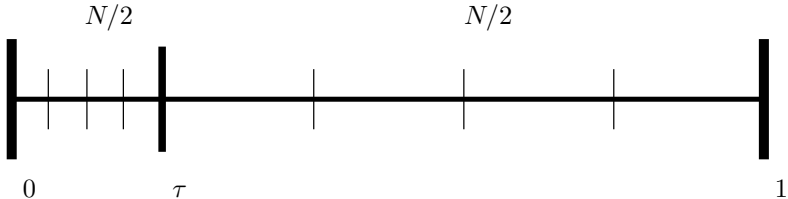
\begin{figure}[htb]
  \begin{center}
    \begin{picture}(0,0)
\includegraphics{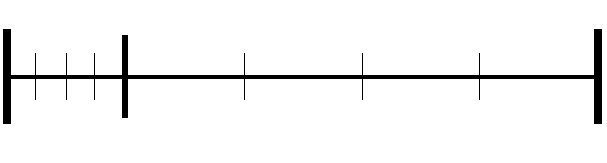}
\end{picture}
\setlength{\unitlength}{4144sp}
\begingroup\makeatletter\ifx\SetFigFont\undefined%
\gdef\SetFigFont#1#2#3#4#5{%
  \reset@font\fontsize{#1}{#2pt}%
  \fontfamily{#3}\fontseries{#4}\fontshape{#5}%
  \selectfont}%
\fi\endgroup%
\begin{picture}(4608,1227)(-53,185)
\put(  1,254){\makebox(0,0)[b]{\smash{{\SetFigFont{10}{12.0}{\familydefault}{\mddefault}{\updefault}{\color[rgb]{0,0,0}$0$}%
}}}}
\put(477,1289){\makebox(0,0)[b]{\smash{{\SetFigFont{10}{12.0}{\familydefault}{\mddefault}{\updefault}{\color[rgb]{0,0,0}$N/2$}%
}}}}
\put(901,254){\makebox(0,0)[b]{\smash{{\SetFigFont{10}{12.0}{\familydefault}{\mddefault}{\updefault}{\color[rgb]{0,0,0}$\tau$}%
}}}}
\put(4501,254){\makebox(0,0)[b]{\smash{{\SetFigFont{10}{12.0}{\familydefault}{\mddefault}{\updefault}{\color[rgb]{0,0,0}$1$}%
}}}}
\put(2746,1289){\makebox(0,0)[b]{\smash{{\SetFigFont{10}{12.0}{\familydefault}{\mddefault}{\updefault}{\color[rgb]{0,0,0}$N/2$}%
}}}}
\end{picture}%

    \caption{Sketch of a Shishkin mesh for the one-dimensional problem~\eqref{eq:1D}}
    \label{fig:1D mesh}
  \end{center}
\end{figure}

\begin{remark}\label{rem:general S mesh}
  \fix{Lin{\ss}~\cite[p10]{Linss10} provides a more general construction of
  a Shishkin mesh  for this problem, introducing a parameter $q$ that
  determines the proportion of mesh points in the layer. Specifically,
  taking $q$ so that $qN$ is an integer,
  one modifies \eqref{eq:tau 1D} so that
$\tau = \min\left\{q, 2\eps/\Clow \ln(N)\right\}$.
Then, one forms uniform grids on
  the subdomains
  $[0,\tau]$ and $[\tau,1]$ so that they have $qN$ and $(1-q)N$
  intervals, respectively. In
  Section~\ref{sec:1D numbers}, we take the standard choice of $q=1/2$; however, in 
  the following section,  we develop the preconditioner for a mesh
  with arbitrarily many mesh points in the layer and interior regions.}
\end{remark}

\subsection{The preconditioner}
\label{sec:1D M}

Consider the one-dimensional problem in \eqref{eq:1D}, discretized
to give the stencil in \eqref{eq:stencil_1d} on a Shishkin mesh.
Within the layer, $h_i \approx \eps/N$ (discarding the $\ln(N)$ factor for simplicity of presentation), so the stencil can be
approximated by
\[
\left[-\frac{N^2}{\eps}, 2\frac{N^2}{\eps} + \frac{c_iN}{\eps} + r_i,
  -\frac{N^2}{\eps} - \frac{c_iN}{\eps}\right],
\]
and we see that both the left- and right-side connections are
significant, with the diffusion terms dominating in the limit of large
N (regardless of the value of $\eps$).  In contrast, in the interior
region, $h_i \approx 1/N$, so the stencil is approximately
\[
\left[-\eps N^2, 2\eps N^2 + c_iN + r_i, -\eps N^2 - c_i N \right].
\]
When $\eps N \ll 1$, the $\Oh(1)$ convection coefficient dominates the
diffusion coefficient, resulting in a system that is dominated by its
upper bidiagonal part, with the entry on the subdiagonal being
comparatively negligible.

Motivated by the above, we consider a block partitioning of the system
matrix, $A$, and associated discrete vectors,
into regions where the meshwidths are $\Oh(\eps/N)$ (typically near
layers, so denoted by $L$), and those where the 
meshwidths are $\Oh(1/N)$  (typically in the domain's interior, so
denoted by $I$). We include the transition point
in the mesh (with an interval
with $\Oh(\eps/N)$ meshwidth to its left and $\Oh(1/N)$ to its right)
in $L$, the layer set.
In this notation, 
\begin{equation}\label{eq:block_A}
A = \begin{bmatrix} A_{LL} & A_{LI} \\ A_{IL} &
  A_{II} \end{bmatrix},
\end{equation}
where $A$ is an \fix{$N\times N$} matrix,
$A_{LL}$ is an $N_L\times N_L$ matrix, and $A_{II}$ being an
$N_I\times N_I$ matrix, with $N = N_L + N_I$.  From above, we see that
we can accurately \fix{approximate} the action of $A_{II}$ by its upper
triangular part (including the diagonal), $M_{II}$, leading to a
block-structured preconditioner for $A$, given as
\begin{equation}
  \label{eq:block_M}
  M = \begin{bmatrix} A_{LL} & A_{LI} \\ A_{IL} &
    M_{II} \end{bmatrix}.
\end{equation}
This can be viewed in several ways, including as a type of Schwarz
iteration where we are, simply, using an inexact subdomain solve on
the interior region of the mesh (noting that $A_{IL}$ has only a
single nonzero entry, in its first row and last column).

\subsection{Theory}\label{sec:1D theory}
Schwarz methods for this problem
have been considered before, in~\cite{MaOR02,CEcheverria_etal_2018a}.
For one-dimensional problems, with tridiagonal discretization
matrices, the spectral structure of the error-propagation operators
often has very tractable form; however, this may limit the
applicability of the resulting preconditioners to one-dimensional
problems, if the resulting methods rely on this special structure.
As we show below, the intuition behind the
block-structured preconditioner in \eqref{eq:block_M} generalizes more
readily.  The analysis of the eigenvalues of the preconditioned
system, however, is somewhat more involved.  We begin by
characterizing the eigenvalues of $M^{-1}A$ based on the matrix
structure, using the notation $\vec{e}^{(k)}$ for the canonical unit
vector of length $N$, with all entries equal to zero except the
$k^\text{th}$, which is equal to one.

\begin{theorem}\label{thm:block_fact}
Let $S_A = A_{II} - A_{IL}A_{LL}^{-1}A_{LI}$ and $S_M = M_{II} -
A_{IL}A_{LL}^{-1}A_{LI}$.  Then $M^{-1}A$ has (at least) $N_L$
eigenvalues equal to 1, with eigenvectors $\vec{e}^{(k)}$ for $1\leq k
\leq N_L$.  The other $N_I$ eigenvalues are the eigenvalues of
$S_M^{-1}S_A$.
\end{theorem}
\begin{proof}
  From the block structure, we can block factorize both $A$ and $M$
  as
  \begin{align*}
    A & = \begin{bmatrix} I
      & 0 \\ A_{IL}A_{LL}^{-1} & I \end{bmatrix}
    \begin{bmatrix} A_{LL} & 0 \\ 0 & S_A \end{bmatrix}
    \begin{bmatrix} I & A_{LL}^{-1}A_{LI} \\ 0 & I \end{bmatrix}, \\
    M & = \begin{bmatrix} I & 0 \\ A_{IL}A_{LL}^{-1} & I \end{bmatrix}
    \begin{bmatrix} A_{LL} & 0 \\ 0 & S_M \end{bmatrix}
    \begin{bmatrix} I & A_{LL}^{-1}A_{LI} \\ 0 & I \end{bmatrix},
  \end{align*}
  where $I$ is understood to be the suitably sized identity matrix.
  By direct calculation, we then have
  \[
M^{-1}A = \begin{bmatrix} I & A_{LL}^{-1}A_{LI} \\ 0 &  I \end{bmatrix}^{-1}
\begin{bmatrix} I & 0 \\ 0 & S_M^{-1}S_A \end{bmatrix}
\begin{bmatrix} I & A_{LL}^{-1}A_{LI} \\ 0 & I \end{bmatrix},
\]
which we recognize as a similarity transformation of the
block-diagonal matrix
\[
\begin{bmatrix} I & 0 \\ 0 & S_M^{-1}S_A \end{bmatrix}.
\]
Thus, $M^{-1}A$ has the same eigenvalues as this matrix, giving (at
least) $N_L$ eigenvalues equal to one, and the remaining eigenvalues
as those of $S_M^{-1}S_A$.  That the unit eigenvalues have
eigenvectors $\vec{e}^{(k)}$ for $1\leq k \leq N_L$ follows from the
fact that $A\vec{e}^{(k)} = M\vec{e}^{(k)}$ for $1\leq k \leq N_L$.
\end{proof}

Computing the eigenvalues of $S_M^{-1}S_A$ is the harder task.  To do
this, we first explicitly compute $A_{IL}A_{LL}^{-1}A_{LI}$, exploiting
the fact that both $A_{IL}$ and $A_{LI}$ have only a single nonzero entry, due
to the tridiagonal structure of $A$.
Noting that, from the block
structure,
\begin{align*}
\left(A_{LI}\right)_{i,j} & = a_{i,N_L+j} \text{ for }1\leq i \leq
N_L, 1\leq j \leq N_I, \\
\left(A_{IL}\right)_{i,j} & = a_{N_L+i,j} \text{ for }1\leq i \leq
N_I, 1\leq j \leq N_L,
\end{align*}
we can recognize that only $\left(A_{LI}\right)_{N_L,1}$ and
$\left(A_{IL}\right)_{1,N_L}$ are nonzero, and we can write
\begin{align*}
A_{LI} & = a^{}_{N_L,N_L+1}\hat{\vec{e}}^{(N_L)}\left(\hat{\vec{e}}^{(1)}\right)^T, \\
A_{IL} & = a^{}_{N_L+1,N_L}\hat{\vec{e}}^{(1)}\left(\hat{\vec{e}}^{(N_L)}\right)^T, 
\end{align*}
where $\hat{\vec{e}}^{(1)}$ denotes the first canonical unit vector of
length $N_I$ and $\hat{\vec{e}}^{(N_L)}$ denotes the last canonical
unit vector of length $N_L$.  Then, from direct calculation, we have
\begin{align*}
A_{IL}A_{LL}^{-1}A_{LI} & =
\left(a^{}_{N_L+1,N_L}\hat{\vec{e}}^{(1)}\left(\hat{\vec{e}}^{(N_L)}\right)^T\right)
A_{LL}^{-1}
\left(a^{}_{N_L,N_L+1}\hat{\vec{e}}^{(N_L)}\left(\hat{\vec{e}}^{(1)}\right)^T\right) \\
& = a^{}_{N_L+1,N_L}a^{}_{N_L,N_L+1}\left(A_{LL}^{-1}\right)_{N_L,N_L}\hat{\vec{e}}^{(1)}\left(\hat{\vec{e}}^{(1)}\right)^T.
\end{align*}
Two general results now enable us to estimate how large of a change this term
represents in $S_M$ and $S_A$.

\begin{lemma}\label{lem:LU_nn}
Let $A=LU$ be the LU factorization of the $n\times n$ matrix, $A$,
with unit diagonal on $L$.  Then $\left(A^{-1}\right)_{n,n} = \left(u_{n,n}\right)^{-1}$.
\end{lemma}
\begin{proof}
Note that $\left(A^{-1}\right)_{n,n}$ is naturally expressed as the
final entry in the vector $A^{-1}\vec{e}^{(n)}$.  From the LU
factorization, $A^{-1}\vec{e}^{(n)} = U^{-1}L^{-1}\vec{e}^{(n)}$.
Since $L$ is lower triangular with unit diagonal, $L^{-1}\vec{e}^{(n)}
= \vec{e}^{(n)}$, so $A^{-1}\vec{e}^{(n)} = U^{-1}\vec{e}^{(n)}$.
Now, since $U$ is upper-triangular, the last entry of
$U^{-1}\vec{e}^{(n)}$ is $u_{n,n}^{-1}$.
\end{proof}

\begin{lemma}\label{lem:diag_dom}
Let $A$ be a tridiagonal and diagonally dominant $n\times n$ matrix
with positive diagonal entries, and let $A=LU$ be its LU factorization
with unit diagonal on $L$.  Then $u_{i,i} \geq |u_{i,i+1}|$ for $1
\leq i \leq n-1$.
\end{lemma}
\begin{proof}
  First consider the LU factorization of $A$, as
  \[
  A = \begin{bmatrix} 1 \\
    \ell_{2,1} & 1 \\
    & \ell_{3,2} & 1 \\
    & & \ddots & \ddots \\
    & & & \ell_{n,n-1} & 1 \end{bmatrix}
  \begin{bmatrix} u_{1,1} & u_{1,2} \\
     & u_{2,2} & u_{2,3} \\
    & & u_{3,3} & u_{3,4} \\
    & & & \ddots & \ddots \\
    & & & & u_{n,n} \end{bmatrix}.
  \]
From here, we can directly calculate that $u^{}_{1,1} = a^{}_{1,1}$,
$u^{}_{1,2} = a^{}_{1,2}$ and, for $i > 1$,
\begin{align*}
  u_{i,i+1} & = a_{i,i+1}, \\
  \ell_{i,i-1} & = a_{i,i-1}/u_{i-1,i-1}, \\
  u_{i,i} & = a_{i,i} - \ell_{i,i-1}u_{i-1,i} = a_{i,i} - a_{i,i-1}\frac{u_{i-1,i}}{u_{i-1,i-1}}.
\end{align*}

Now, consider a proof of the theorem by induction.  For the base case,
we have $u_{1,1} \geq |u_{1,2}|$ from the original assumption on
diagonal dominance of $A$.  For the inductive step, assume
$u_{i-1,i-1} \geq |u_{i-1,i}|$.  Then,
\[
u_{i,i} \geq a_{i,i} -
|a_{i,i-1}|\left|\frac{u_{i-1,i}}{u_{i-1,i-1}}\right| \geq a_{i,i} -
|a_{i,i-1}| \geq |a_{i,i+1}|,
\]
where the last step follows by diagonal dominance of $A$.  Since
$u_{i,i+1} = a_{i,i+1}$, this completes the inductive step and the
proof.
\end{proof}

\begin{corollary}
  \label{cor:bound}
Let $A$ be the $N\times N$ discretization matrix of
\eqref{eq:1D}, as given in \eqref{eq:stencil_1d}, partitioned as
in \eqref{eq:block_A}.  Then,
\[
\left(A_{IL}A_{LL}^{-1}A_{LI}\right)_{1,1} \leq |a^{}_{N_L+1,N_L}| = \left(A_{IL}\right)_{1,N_L}.
\]
\end{corollary}
\begin{proof}
  From above, we have that
  \[
  \left(A_{IL}A_{LL}^{-1}A_{LI}\right)_{1,1} = a_{N_L+1,N_L}a_{N_L,N_L+1}\left(A_{LL}^{-1}\right)_{N_L,N_L}.
  \]
  Lemma~\ref{lem:LU_nn} shows that if $A_{LL} = LU$ is the LU
  factorization of $A_{LL}$ (with unit diagonal on $L$), then
  $\left(A_{LL}^{-1}\right)_{N_L,N_L} =
  \left(u_{N_L,N_L}\right)^{-1}$, and
  \begin{equation}\label{eq:bound_schur}
  \left(A_{IL}A_{LL}^{-1}A_{LI}\right)_{1,1} = a_{N_L+1,N_L}\frac{a_{N_L,N_L+1}}{u_{N_L,N_L}}.
  \end{equation}

  Extending the induction argument from Lemma~\ref{lem:diag_dom}, we
  have that
  \[
u_{N_L,N_L} = a_{N_L,N_L} -
a_{N_L,N_L-1}\frac{u_{N_L-1,N_L}}{u_{N_L,N_L}} \geq a_{N_L,N_L} -
|a_{N_L,N_L-1}|.
\]
From the definition of the stencil in \eqref{eq:stencil_1d}, this
gives $u_{N_L,N_L} \geq |a_{N_L,N_L+1}|$.  Using this to bound the
right-hand side of \eqref{eq:bound_schur} gives the stated result.
\end{proof}

We are now ready to state and prove our main result, on the
eigenvalues of the preconditioned system.

\begin{theorem}
  \label{thm:big_bound}
Let $A$ be the $N\times N$ tridiagonal matrix given by the stencil in
\eqref{eq:stencil_1d}, block partitioned as in \eqref{eq:block_A} and
let $M$ be the preconditioner defined in \eqref{eq:block_M}.  Assume
that there is a constant, $\alpha$, such that the discretization mesh
satisfies $\alpha/N \leq h_i$ for $N_L+1 \leq i \leq N$.
Then, any eigenvalue, $\lambda$, of
$M^{-1}A$ satisfies
\begin{equation}
  \label{eq:big_bound}
1-\frac{8\eps N}{\Clow\alpha} \leq \lambda \leq 1,
\end{equation}
where $\Clow > 0$ is defined in~\eqref{eq:coef assumptions}.
\end{theorem}
\begin{proof}
From Theorem~\ref{thm:block_fact}, we have that all of the eigenvalues
of $M^{-1}A$ are either $1$ (and trivially satisfy the bound) or are
eigenvalues of $S_M^{-1}S_A$.  Note that we can write $S_A = S_M -
\hat{L}$, where $\hat{L}$ is a strictly lower-triangular matrix, so
that
\[
S_M^{-1}S_A = S_M^{-1}\left(S_M - \hat{L}\right)= I - S_M^{-1}\hat{L}.
\]
Consequently, any eigenvalue, $\lambda$, of $S_M^{-1}S_A$ can be
written as $\lambda = 1-\gamma$, where $\gamma$ is an eigenvalue of
$S_M^{-1}\hat{L}$.  Note, also, that this preserves algebraic
and geometric multiplicities of the eigenvalues.

Now, we recognize $S_M^{-1}\hat{L}$ as the iteration matrix for a
reverse-ordered Gauss-Seidel iteration on $S_A$, and that since $S_A$
is tridiagonal, it is a 2-cyclic matrix.  Thus, by classical arguments
\cite{RSVarga_2000}, we have that any eigenvalue, $\gamma$, of
$S_M^{-1}\hat{L}$ is either zero or the square of an eigenvalue of the
corresponding Jacobi iteration matrix.  Here, we recognize that this
relationship does \textit{not} preserve geometric multiplicities, so
that $\gamma = 0$ may correspond to an eigenvalue whose algebraic
multiplicity may be larger than its geometric multiplicity.  Writing
$D_A$ as the diagonal matrix whose entries match those of $S_A$, we
can naturally write
\[
I - D_A^{-1}S_A = \begin{bmatrix}
    0 & b_1 \\
    b_{-1} & 0 & b_2 \\
    & b_{-2} & 0 & b_3 \\
    & & \ddots & \ddots & \ddots \\
    & & & b_{2-N_I} & 0 & b_{N_I-1} \\
    & & & & b_{1-N_I} & 0 \end{bmatrix},
\]
with natural definitions of $b_k$ and $b_{-k}$ for $1 \leq k \leq
N_I-1$, noting that $S_A$ differs from $A_{II}$ only in its first
entry.  We can define the diagonal scaling matrix $\Sigma$, by taking
$\sigma_{1,1} = 1$, then defining $\sigma_{k+1,k+1} =
\sqrt{b_k/b_{-k}}\sigma_{k,k}$ for $2\leq k \leq N_I-1$, noting that both
$b_k$ and $b_{-k}$ are positive.  By direct calculation, we then have
the similarity transform
\[
\Sigma\!\left(I - D_A^{-1}S_A\!\right)\!\Sigma^{-1} \! = \! \begin{bmatrix}
    0 & \sqrt{b_{-1}b_1} \\
    \sqrt{b_{-1}b_1} & 0 & \sqrt{b_{-2}b_2} \\
    & \sqrt{b_{-2}b_2} & 0 & \ddots \\
    & & \qquad \ddots &  \\
    & & \ddots & 0 & \sqrt{b_{1-N_I}b_{N_I-1}} \\
    & & & \sqrt{b_{1-N_I}b_{N_I-1}} & 0 \end{bmatrix}\!\! .
\]
This is symmetric, so Ger\v{s}gorin's Theorem implies that any
eigenvalue, $\mu$, of $I - D_A^{-1}S_A$ must be real and satisfy
  \[
|\mu| \leq 2\max_{1\leq k \leq N_I-1} \sqrt{b_{-k}b_k}.
\]

Considering the stencil in \eqref{eq:stencil_1d}, for $2\leq k \leq
N_I-1$, we have
\begin{align*}
  b_{-k} & = \frac{-a_{N_L+k+1,N_L+k}}{a_{N_L+k+1,N_L+k+1}} \\
  & = \frac{\frac{\eps}{\bar{h}_{N_L+k+1}h_{N_L+k+1}}}{\frac{\eps}{\bar{h}_{N_L+k+1}}\left(\frac{1}{h_{N_L+k+1}}+\frac{1}{h_{N_L+k+2}}\right) + \frac{c_{N_L+k+1}}{h_{N_L+k+2}} + r_{N_L+k+1}} \leq
    \frac{2\eps}{\Clow h_{N_L+k+1}} \leq \frac{2\eps N}{\Clow\alpha},\\
    b_k & = \frac{-a_{N_L+k,N_L+k+1}}{a_{N_L+k,N_L+k}} \\
    & = \frac{\frac{\eps}{\bar{h}_{N_L+k}h_{N_L+k+1}} +
      \frac{c_{N_L+k}}{h_{N_L+k+1}}}{\frac{\eps}{\bar{h}_{N_L+k}}\left(\frac{1}{h_{N_L+k}}+\frac{1}{h_{N_L+k+1}}\right)
      + \frac{c_{N_L+k}}{h_{N_L+k+1}} + r_{N_L+k}} \leq 1
\end{align*}
The same bound is true for $b_{-1}$.  For the remaining term, we use
the bound in Corollary~\ref{cor:bound} to get
\[
b_1 \leq
\frac{\frac{\eps}{\bar{h}_{N_L+1}h_{N_L+2}} +
      \frac{c_{N_L+1}}{h_{N_L+2}}}{\frac{\eps}{\bar{h}_{N_L+1}}\left(\frac{1}{h_{N_L+1}}+\frac{1}{h_{N_L+2}}\right)
  + \frac{c_{N_L+1}}{h_{N_L+2}} + r_{N_L+1} - \frac{\eps}{\bar{h}_{N_L+1}h_{N_L+1}}}
\leq 1.
\]
Taken together, these show that any eigenvalue, $\mu$, of the Jacobi
iteration matrix for $S_A$ satisfies the bound that
\[
|\mu| \leq 2\sqrt{\frac{2\eps N}{\Clow\alpha}}.
\]
Thus, the eigenvalues, $\gamma$, of $S_M^{-1}\hat{L}$ satisfy the
bound
\[
0 \leq \gamma \leq \frac{8\eps N}{\Clow\alpha}.
\]
Noting that the eigenvalues of $M^{-1}A$ are either 1 or $1-\gamma$
for an eigenvalue of $S_M^{-1}\hat{L}$ completes the proof.
\end{proof}

\fix{The sharpness of Theorem~\ref{thm:big_bound} is investigated
  further in Section~\ref{sec:1D numbers}.}
\begin{remark}
  \label{rem:not just S-meshes}
We note that Theorem~\ref{thm:big_bound} makes no assumptions on the
meshwidths in the layer region.  In fact, the theorem applies equally
well to a discretization on a uniform mesh (where it shows that
appropriately ordered Gauss-Seidel yields an effective stationary
iteration in the singularly perturbed limit).  Additionally, the
theorem covers both cases of piecewise uniform (Shishkin) or graded
(e.g., Bakhvalov) meshes.
\end{remark}

\begin{remark}
  \label{rem:small epsilon}
  The lower bound in \eqref{eq:big_bound} is useful only when $\eps$ is
  sufficiently small, relative to $N$, specifically, when
  \begin{equation}\label{eq:small epsilon*N}
    \eps N \leq \Clow \alpha/{8}.
  \end{equation}
  This is not a significant restriction; for larger
$\eps$, a uniform mesh is sufficient to resolve all aspects of the
solution. Furthermore, as per the discussion in
Section~\ref{sec:discretization}, if $\eps N$ is so large that
\eqref{eq:small epsilon*N} does not hold, a discretization using
central differences is stable on a uniform mesh, and specialized preconditioners are not needed.  
\end{remark}

\subsection{Numerical Experiments}\label{sec:1D numbers}

In this section, we verify the robustness of the bounds presented in
\autoref{thm:big_bound}, and then investigate the practical
issue of determining a suitable stopping criterion when GMRES is
preconditioned with $M$ as given in \eqref{eq:block_M}.

Our test problem is 
\begin{equation}
  \label{eq:1D example}
  -\eps u'' - \big(2+\sin(5x)\big) u' + u  = 4 e^{-x} \text{ on } (0,1),
  \quad u(0)=u(1)=0.
\end{equation}
A computed solution when $\eps=10^{-2}$ is shown in \autoref{fig:1D example}.
\begin{figure}[htb]
  \begin{center}
  \includegraphics[width=0.5\textwidth]{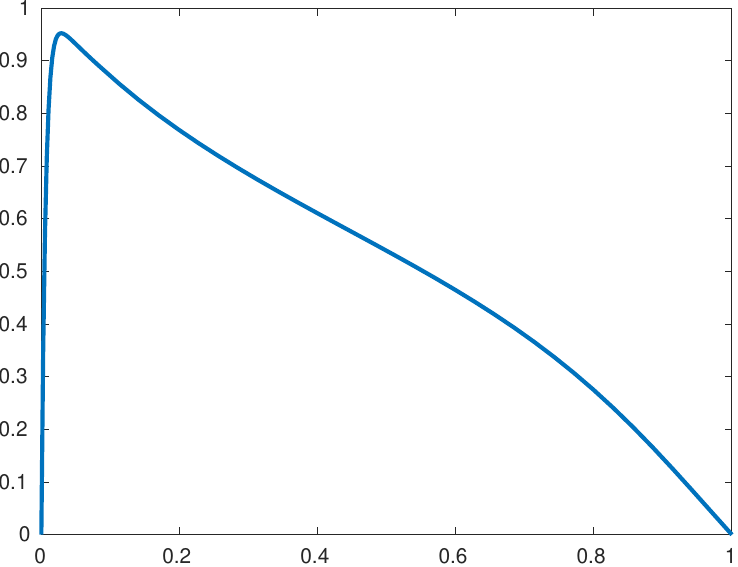}
  \caption{Solution to \eqref{eq:1D example} with $\eps=10^{-2}$}
  \label{fig:1D example}
  \end{center}
\end{figure}

In  \autoref{tab:1D errors}, we verify that the error bounds reported
in \eqref{eq:error bound 1D Shishkin}
are sharp: the upwind scheme
applied on the Shishkin mesh of Section~\ref{sec:1D S-mesh} yields a
solution with error that is  bounded
independently of $\eps$.
\fix{We also show $\rho^N$, the estimated rate of convergence for
  the smallest value of $\eps$. It is in agreement with
   \eqref{eq:error bound 1D Shishkin}: the error is proportional to $N^{-1}\ln N$.}

Since an analytical solution to
\eqref{eq:1D example} is not available, for each $N$,
these errors are estimated by comparing with a benchmark solution
computed on a mesh with the same transition points, but $64N$ mesh
intervals.

\begin{table}[htb]
  \begin{center}
  \caption{Error, in the discrete maximum norm, for \eqref{eq:1D
      example}.}
  \label{tab:1D errors}
\setlength{\tabcolsep}{3pt}
\begin{tabular}{c||c|c|c|c|c}
  \hline
  $\eps$   &  $N= 128$  &  $N= 256$  &  $N= 512$  &  $N=1024$  &  $N=2048$  \\ \hline

$1$ & $2.425\times 10^{-3}$ & $1.220\times 10^{-3}$ & $6.120\times 10^{-4}$ & $3.065\times 10^{-4}$ & $1.534\times 10^{-4}$ \\ 
$10^{-1}$ & $2.725\times 10^{-2}$ & $1.409\times 10^{-2}$ & $7.173\times 10^{-3}$ & $3.619\times 10^{-3}$ & $1.818\times 10^{-3}$ \\ 
$10^{-2}$ & $4.963\times 10^{-2}$ & $3.007\times 10^{-2}$ & $1.742\times 10^{-2}$ & $9.851\times 10^{-3}$ & $5.473\times 10^{-3}$ \\ 
$10^{-3}$ & $4.822\times 10^{-2}$ & $2.927\times 10^{-2}$ & $1.699\times 10^{-2}$ & $9.627\times 10^{-3}$ & $5.357\times 10^{-3}$ \\ 
$10^{-4}$ & $4.800\times 10^{-2}$ & $2.914\times 10^{-2}$ & $1.692\times 10^{-2}$ & $9.586\times 10^{-3}$ & $5.334\times 10^{-3}$ \\ 
$10^{-5}$ & $4.798\times 10^{-2}$ & $2.913\times 10^{-2}$ & $1.691\times 10^{-2}$ & $9.582\times 10^{-3}$ & $5.332\times 10^{-3}$ \\ 
$10^{-6}$ & $4.798\times 10^{-2}$ & $2.912\times 10^{-2}$ & $1.691\times 10^{-2}$ & $9.581\times 10^{-3}$ & $5.332\times 10^{-3}$ \\ 
$10^{-7}$ & $4.798\times 10^{-2}$ & $2.912\times 10^{-2}$ & $1.691\times 10^{-2}$ & $9.581\times 10^{-3}$ & $5.332\times 10^{-3}$ \\ 
$10^{-8}$ & $4.798\times 10^{-2}$ & $2.912\times 10^{-2}$ & $1.691\times 10^{-2}$ & $9.581\times 10^{-3}$ & $5.332\times 10^{-3}$ \\ 
  \hline
  $\rho^N$ & 0.678 &    0.720 &    0.784 &    0.820 &    0.846  \\ \hline
  \end{tabular}
  \end{center}
\end{table}

\fix{Our numerical experiments have verified that Theorem~\ref{thm:big_bound} is correct and quite sharp.
For the data corresponding to the first column of \autoref{tab:1D
  errors}, and denoting (a numerical estimate for) the smallest eigenvalue
of $M^{-1}A$ as  $\lambda_{\min}$,
we verify that $\lambda_{\min} \geq 1-8\eps N/(\Clow \alpha)$ by
plotting $1-\lambda_{\min}$ and $8\eps N/(\Clow \alpha)$, where the
largest valid value of $\alpha$ is taken.
Indeed, for the range of $N$ and $\eps$ reported
in~\autoref{tab:1D errors}, 
we observe that $\lambda_{\min}$ is found between
$1-4\eps N/(\Clow\alpha)$ and
$1-\eps N/(2\Clow\alpha)$, 
as long as~\eqref{eq:small epsilon*N} holds.}

\begin{figure}[htb]
  \begin{center}

  \includegraphics[width=0.5\textwidth]{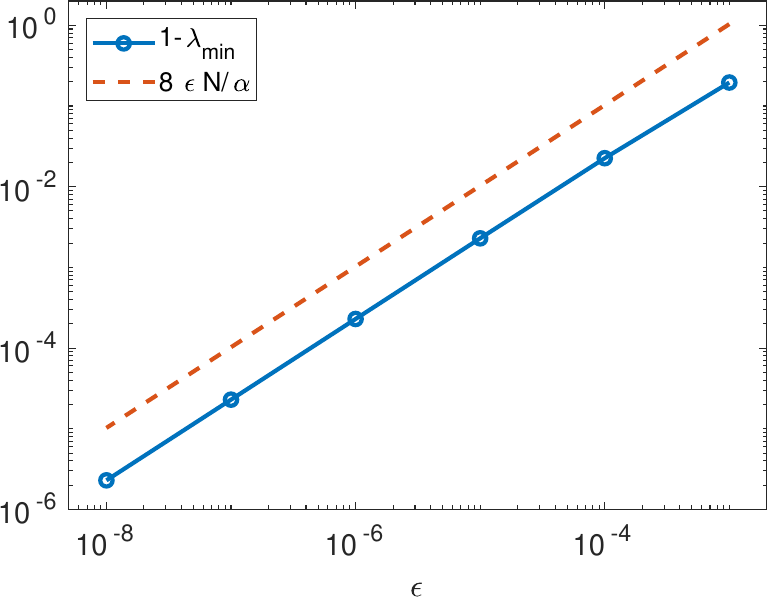}
  \caption{\fix{Comparison of $1 - \lambda_{\min}$ and $8 \eps N /(\Clow
    \alpha)$, where $\lambda_{\min}$ is the smallest eigenvalue of
    $M^{-1}A$, for $N=128$, and $\eps = 10^{-3}, 10^{-4}, \dots, 10^{-8}$.}}
  \label{fig:eig bounds}
  \end{center}
\end{figure}

\subsubsection{Stopping criterion}
Let $U^E$ be the exact solution to the linear system arising from the
scheme of Section~\ref{sec:discretization}, and $U^{(k)}$ be the $k$th
iterate computed by an iterative solver. Let $g(N)$ be the expected
discretization error; for example, for the mesh and method we
consider here,
\[
  g(N):=\|u-U^E\|_{\infty}\le CN^{-1}\ln N.
\]
(See, e.g., \cite{NV19}, for analysis of a mesh where fully first-order convergence is expected.)
We wish to iterate until 
\[
\|u- U^{(k)}\|_{\infty}\le \|u-U^E\|_{\infty} + \| U^E-U^{(k)}\|_{\infty} \le C g(N)
\]
where $C$ is a moderate constant. Of course, we cannot compute
$E^{(k)}=U^E-U^{(k)}$, but we can compute
the residual $R^{(k)}=F-AU^{(k)}=AE^{(k)}$. Therefore we write
$E^{(k)}=A^{-1}R^{(k)}$, giving 
\[
\|U^E-U^{(k)}\|_{\infty}\le \|A^{-1}\|_{\infty}\|R^{(k)}\|_{\infty}.
\]
In contrast to the stopping criterion proposed in
\cite{SMacLachlan_NMadden_2013a} for reaction-diffusion problems, in
which $\|A^{-1}\|_{\infty}$ is unbounded when $\eps\to 0$, for
convection-diffusion problems, the system matrix $A$ defined
in~\eqref{eq:block_A} is an M-matrix.  Thus, it is easy to verify that
(see, for example, \cite{NSV18}) 
\[
\|A^{-1}\|_{\infty}\le C.
\] 
Therefore, we iterate until  $\|R^{(k)}\|_{\infty}\le Kg(N)$,  for
some user-chosen parameter $K$. Numerical experience suggests that
taking $K = \|U^E\|_\infty$ is  reasonable (i.e., an $\mathcal{O}(1)$ value).

\subsubsection{Performance of the preconditioner}
The discretization of \eqref{eq:1D example} leads to a tridiagonal
system which is easily solved using direct methods, even for very
large values of the discretization parameter, $N$. However, it is
instructive to consider the performance of an iterative solver for
this problem, when preconditioned with $M$ as defined
in~\eqref{eq:block_M}.

To that end, in \autoref{tab:1D iteration counts}, we report the number
of iterations required by the MATLAB \texttt{gmres}
function~\cite{Walker88}, modified slightly to implement the stopping
criterion, and with no restarts. The results are for those values of 
$\eps$ and $N$ included in \autoref{tab:1D errors} for which 
\eqref{eq:small epsilon*N} holds.
They show that, as expected, few
iterations are required as $\eps \to 0$. The greatest number of
iterations are required for largest reported values of $\eps$ and $N$,
where, although \eqref{eq:small epsilon*N} holds, 
one has that
$\eps N > \Clow \alpha /2$. In all other cases,
$\eps N \ll \Clow \alpha$, and few iterations are required to ensure
convergence. For example, if $\eps N \leq 0.01$, then no more than 4
iterations are required in any case.

\begin{table}[htb]
  \caption{Iteration counts for GMRES preconditioned with $M$ in~\eqref{eq:block_M}}
  \label{tab:1D iteration counts}
\begin{center}
  \begin{tabular}{c||c|c|c|c|c|c}
    \hline
$\eps$   &  $N=128$  &  $N=256$  &  $N=512$  &  $N=1024$  &  $N=2048$  \\ \hline

$10^{-3}$ &         4 &         -- &        -- &        -- &       -- \\ 
$10^{-4}$ &         2 &         4 &         6 &        14 &        38 \\ 
$10^{-5}$ &         1 &         2 &         3 &         5 &         9 \\ 
$10^{-6}$ &         1 &         1 &         2 &         2 &         4 \\ 
$10^{-7}$ &         1 &         1 &         1 &         2 &         2 \\ 
$10^{-8}$ &         1 &         1 &         1 &         1 &         2
      \\  \hline
    \end{tabular}
\end{center}
\end{table}

\section{Two-dimensional problems}\label{sec:2D}
We now consider the two-dimensional problem on the unit square given
in \eqref{eq:2D}, with homogeneous Dirichlet boundary conditions on
all four sides and $r(x,y) \geq 0$.  We focus on two cases, both where
$\vec{c}(x,y)$ is componentwise non-negative.  In the first case, we
fix $c_2(x,y) = 0$ and require $c_1(x,y) > 0$.  From the standard
theory of convection-diffusion problems, solutions to \eqref{eq:2D} in
this case are expected to exhibit parabolic (characteristic)  boundary
layers of width $\Oh(\sqrt{\eps}\ln(1/\eps))$ along $y=0$ and $y=1$,
and a single exponential boundary layer of width
$\Oh(\eps\ln(1/\eps))$ along $x=0$.  We will focus our discussion on
the case where the forcing function is \textit{compatible} with the
boundary conditions, so that no layer forms along $y=1$, noting that
this is solely for convenience and that all constructions could be
directly extended to handle the case of two parabolic layers.
The second case that we consider is when both $c_1(x,y) > 0$ and
$c_2(x,y) > 0$, which leads to the formation of two exponential layers
in the solution, along $x=0$ and $y=0$.

For both problems, we make use of tensor-product Shishkin meshes for
the discretization, now defining separate transition points in the
$x$- and $y$-directions, denoted by $\tau_x$ and $\tau_y$,
respectively.  A sketch of such a mesh for the case with one parabolic
and one exponential layer is shown in Figure \ref{fig:2D mesh}.  We make the following
standard choices for the transition points on an $N\times N$ mesh for
the first case, where we assume $0 < \Clow < c_1(x,y)$ for all $(x,y)
\in (0,1)^2$,

\begin{equation}\label{eq:tau parabolic}
  \tau_x  =
  \min\left\{\frac{1}{2},\frac{\sigma\eps}{\Clow}\ln(N)\right\},
  \quad \text{ and } \quad
  \tau_y  = \min\left\{\frac{1}{2},\sigma\sqrt{\eps}\ln(N)\right\}.
\end{equation}
Here, we take $\sigma = 5/2$ as a value that is at least as large as
the order of the discretization scheme discussed below.  For the case
of two exponential layers, we assume that both $0 < \Clow_1 <
c_1(x,y)$ and $0 < \Clow_2 < c_2(x,y)$, and 
take

\begin{equation}\label{eq:tau exponential}
  \tau_x  =
  \min\left\{\frac{1}{2},\frac{\sigma\eps}{\Clow_1}\ln(N)\right\}
    \quad \text{ and } \quad
  \tau_y  = \min\left\{\frac{1}{2},\frac{\sigma\eps}{\Clow_2}\ln(N)\right\}.
\end{equation}
In both cases, we then form a standard tensor-product Shishkin mesh,
by first dividing the unit interval on the $x$-axis into $N/2$
equal-sized intervals from $0$ to $\tau_x$ and $N/2$ equal-sized
intervals from $\tau_x$ to $1$ to form the mesh $\Omega_x^N$, then
dividing the unit interval on the $y$-axis into $N/2$ equal-sized
intervals from $0$ to $\tau_y$ and $N/2$ equal-sized intervals from
$\tau_y$ to $1$ to form the mesh $\Omega_y^N$ and, finally, forming
the standard (quadrilateral) tensor-product mesh, $\Omega_x^N \times
\Omega_y^N$.  While we focus on the Shishkin case below, we note that
the ideas developed could equally-well be applied to many fitted
tensor-product mesh constructions appropriate for such singularly
perturbed problems; the key idea that is required for what follows is
the ability to identify a transition point between the ``interior''
region of the mesh, where meshwidths are bounded below by an
$\Oh(1/N)$ value, and the layer regions, where meshwidths may be much
smaller.

\begin{figure}[htb]
  \begin{center}
    \begin{picture}(0,0)%
\includegraphics{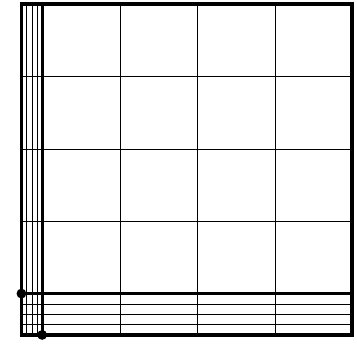}%
\end{picture}%
\setlength{\unitlength}{1450sp}%
\begingroup\makeatletter\ifx\SetFigFont\undefined%
\gdef\SetFigFont#1#2#3#4#5{%
  \reset@font\fontsize{#1}{#2pt}%
  \fontfamily{#3}\fontseries{#4}\fontshape{#5}%
  \selectfont}%
\fi\endgroup%
\begin{picture}(7730,7775)(-464,-6871)
\put(-449,-5506){\makebox(0,0)[b]{\smash{{\SetFigFont{11}{13.2}{\rmdefault}{\mddefault}{\updefault}{\color[rgb]{0,0,0}$\tau_y$}%
}}}}
\put(586,-6811){\makebox(0,0)[b]{\smash{{\SetFigFont{11}{13.2}{\rmdefault}{\mddefault}{\updefault}{\color[rgb]{0,0,0}$\tau_x$}%
}}}}
\end{picture}%

  \caption{Sketch of a tensor-product Shishkin mesh for the case of a
    parabolic layer along the edge $y=0$ and an exponential layer
    along the edge $x=0$.}
  \label{fig:2D mesh}
  \end{center}
\end{figure}
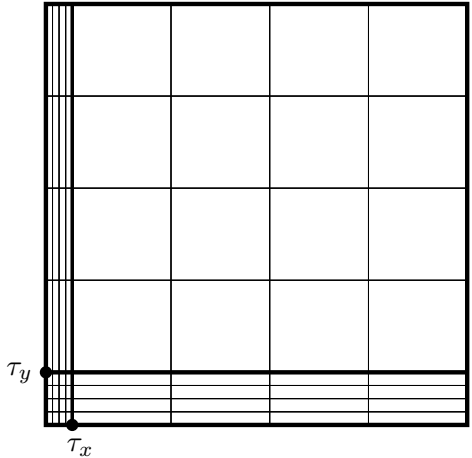

On such a mesh, we make use of a standard upwind finite-difference
discretization for \eqref{eq:2D} on a non-uniform mesh.  At a
mesh point $(x_i,y_j)$ for $1 \leq i,j \leq N$, we define $h_i = x_i -
x_{i-1}$ and $k_j = y_j - y_{j-1}$, with $h_{i+1}$ and $k_{j+1}$
defined similarly, and $\overline{h}_i = (h_i + h_{i+1})/2$ and
$\overline{k}_j = (k_j+k_{j+1})/2$.  The discretization then takes the
same pattern as a standard 5-point finite-difference operator, with
values
\[
\left( \begin{array}{ccc}
  & \frac{-\eps}{\overline{k}_jk_{j+1}} - \frac{c_{2,ij}}{k_{j+1}} \\
  \frac{-\eps}{\overline{h}_ih_i} &
  \frac{\eps}{\overline{h}_i}\left(\frac{1}{h_i}+\frac{1}{h_{i+1}}\right)
  +
  \frac{\eps}{\overline{k}_j}\left(\frac{1}{k_j}+\frac{1}{k_{j+1}}\right)
  + \frac{c_{1,ij}}{h_{i+1}} + \frac{c_{2,ij}}{k_{j+1}} + r_{ij} &
  \frac{-\eps}{\overline{h}_ih_{i+1}} - \frac{c_{1,ij}}{h_{i+1}} \\
   & \frac{-\eps}{\overline{k}_jk_{j}} \end{array}\right),
\]
where $c_{1,ij} = c_1(x_i,y_j)$, $c_{2,ij} = c_2(x_i,y_j)$, and
$r_{ij} = r(x_i,y_j)$.  We note that the upwind finite-difference
discretization results in a discretization matrix that is both
irreducibly diagonally dominant and an M-matrix, since we have $c_{1,ij} > 0$,
$c_{2,ij} \geq 0$, and $r_{ij} \geq 0$.

\subsection{Preconditioner construction}\label{sec:2D M}

In order to develop the preconditioner, we consider reordering and
partitioning the discretization matrix, $A$, into block four-by-four structure,
writing
\[
A = \left[ \begin{array}{cccc} A_{CC} & A_{CX} & A_{CY} & A_{CI} \\
    A_{XC} & A_{XX} & A_{XY} & A_{XI} \\
    A_{YC} & A_{YX} & A_{YY} & A_{YI} \\
    A_{IC} & A_{IX} & A_{IY} & A_{II} \end{array}\right],
\]
where we use the subscripts
  $C$ to denote the corner region, i.e., the mesh points in the rectangle $[0,\tau_x] \times [0,\tau_y]$, 
  $Y$ to denote the region  $(\tau_x,1] \times [0,\tau_y]$, 
  $X$ to denote the region  $[0, \tau_x] \times (\tau_y,1]$,  and
  $I$ to denote the interior region, i.e.,  $(\tau_x,1] \times (\tau_y,1]$.
  As above, we include the transition points in the edge
and corner regions, with the corner region including points with both
$x_i=\tau_x$ and $y_j = \tau_y$, while the edge regions include just
the transition points adjacent to the interior region.  For simplicity
in explanation, we assume that grid points in each block are ordered
lexicographically by index in the mesh, from their lower-left corners
to their upper-right corners.

In both cases under consideration,
we have a convective term in our PDE that ``pushes'' information from
right-to-left and, in the case of two exponential layers, from
top-to-bottom.  Ordering the discretization as above, a natural
structure for a preconditioner, then, is as a block upper-triangular
matrix,
\begin{equation}
  \label{eq:2D blocks}
M = \left[ \begin{array}{cccc} M_{CC} & A_{CX} & A_{CY} & A_{CI} \\
    0 & M_{XX} & A_{XY} & A_{XI} \\
    0 & 0 & M_{YY} & A_{YI} \\
    0 & 0 & 0 & M_{II} \end{array}\right],
\end{equation}
since solution of $Mz=r$ then propagates information from the interior
region to all three other regions, while information from the two edge
regions is propagated to the corner region, matching the natural
convective structure of the system.  Since the off-diagonal blocks of
the matrix are only needed for matrix-vector products to propagate
information from solves with the diagonal blocks, there is no
computational advantage to approximating these.  In contrast, we
consider in detail how to best approximate each of the diagonal blocks
so that the overall cost of solving the linear system $Mz=r$ is only
$\Oh(N^2)$, matching the asymptotic cost of a matrix-vector
multiplication with the system matrix, $A$.

We treat the four diagonal blocks of $M$ in~\eqref{eq:2D blocks}
separately, as follows.
\begin{description}
\item [$M_{II}$:] This block is associated with points in the interior region
  $(\tau_x, 1]\times (\tau_y,1]$. Around a mesh point in this region, we have $h_i,
h_{i+1} \geq C/N$ as well as $k_j, k_{j+1} \geq C/N$, thus, the diffusion terms in the stencil are of size $\Oh(\eps
N^2)$ while the convection terms are $\Oh(N)$ in size.  Under the
typical assumption that $\eps N \ll 1$, this says that the convection
term(s) dominate, and that a good approximation of $A_{II}$ is by
its upper-triangular part, resulting in a downstream Gauss-Seidel
approximation that sweeps from the upper-right corner of the interior
region to the bottom-left corner.  In the case where there is a
parabolic layer, the structure is even simpler, since $c_{2,ij} = 0$
and the system is dominated by only its diagonal and one off-diagonal
term.  Here, any Gauss-Seidel ordering that sweeps from right-to-left
would be acceptable, but it is simpler to use the same ordering in
both cases.  We note the cost of a solve with $M_{II}$ is bounded by
that of a matrix-vector multiplication with $A_{II}$, achieving our
cost goal.

\item [$M_{YY}$:] This block is associated with points in the region
  $(\tau_x, 1]\times [0,\tau_y]$. At a mesh point here we have $h_i,
 h_{i+1} \geq C/N$, while we have much smaller values for
$k_j$ and $k_{j+1}$.  In the case of a classical Shishkin mesh for a
parabolic boundary layer in this region, for
example, we have $k_j = k_{j+1} \approx 2\sqrt{\eps}\ln(N)/N$, while for a
classical Shishkin mesh for an exponential boundary layer in this region (when
$c_{2,ij} \neq 0$), we have $k_j = k_{j+1} \approx 2\eps\ln(N)/N$.  In the
case of a parabolic layer, 
only one off-diagonal term (that to the ``West'', from $(x_i,y_j)$ to
$(x_{i-1},y_{j-1})$ is asymptotically smaller than the rest, with the
meshwidth in the $y$-direction leading to off-diagonal entries of size
$\Oh(N^2/\ln^2(N))$ in the ``North'' and ``South'' directions, while the
convection term leads to the ``East'' off-diagonal entry having size
$\Oh(N)$.  In order to account for these three terms, we approximate
$A_{YY}^{-1}$ by a
block downstream Gauss-Seidel approximation, where we use line-solves
along lines of constant $x$-coordinate, ordered from right-to-left.
These line solves are implemented using Thomas' algorithm, which gives
$\Oh(N)$ cost to each solve and an $\Oh(N^2)$ cost to the inversion
of $M_{YY}$ constructed in this way.  For the case of two exponential
layers, the diffusion terms in the $y$-direction become dominant, of
size $\Oh(N^2/(\eps\ln^2(N))$, again prompting the use of line solves along
lines of constant $x$-coordinate to approximate $A_{YY}^{-1}$.  While the
ordering is less important here, we keep the right-to-left ordering
for simplicity.

\item [$M_{XX}$:] this block is associated with points in the region
  $[0,\tau_x]\times (\tau_y,1]$.  In contrast to the mesh points in the
  region associated with $M_{YY}$, here
we have $k_j, k_{j+1} \geq C/N$, while we have much smaller
values for $h_i$ and $h_{i+1}$ in order to resolve the exponential
boundary layer in the solution at $x=0$.  On a classical Shishkin
mesh, for example, we expect $h_i = h_{i+1} \approx 2\eps\ln(N)/N$.  This
results in relatively small contributions to the matrix from the
diffusion terms in the $y$-direction, which are of size $\Oh(\eps
  N^2)$.  In comparison, the diffusion terms in the $x$-direction are
of size $\Oh(N^2/(\eps\ln^2(N)))$, while the convection terms in the
$x$-direction are of size $\Oh(N/(\eps\ln(N)))$.  This motivates
approximating $A_{XX}^{-1}$ using line solves along lines of constant
$y$-coordinate.  While the North coefficient is never large, the case
of two exponential layers gives a $y$-direction convection
contribution of size $\Oh(N)$; thus, we perform these line solves
sequentially, sweeping from the top of the mesh downwards, to resolve
the convection in the downward direction.  As in the $Y$ region, the cost of each line solve is $\Oh(N)$ operations, and we
perform $\Oh(N)$ of them, giving a total cost of inverting $M_{XX}$
that is still $\Oh(N^2)$.

\end{description}
We devote Section~\ref{sec:multigrid in the corner} to 
the approximation in the corner region of the mesh, but first 
pause to consider a simple bound on the convergence rate of an
idealized form of the preconditioner.  Using the approximations above,
we can define
\[
\hat{M} = \left[ \begin{array}{cccc} A_{CC} & A_{CX} & A_{CY} & A_{CI} \\
    0 & M_{XX} & A_{XY} & A_{XI} \\
    0 & 0 & M_{YY} & A_{YI} \\
    0 & 0 & 0 & M_{II} \end{array}\right],
\]
and consider the classical theory of regular splittings
\cite[\S 3.6]{RSVarga_2000}.  Since $A$ is an irreducibly diagonally
dominant M-matrix, the splitting of $A = \hat{M} - \hat{N}$ is a
regular splitting (since $\hat{M}$ inherits the property of being an
M-matrix by its construction from $A$ \cite[Theorem
  3.25]{RSVarga_2000}, and the implicit definition of $\hat{N} =
\hat{M} - A$ yields a component-wise non-negative matrix).  Thus,
\cite[Theorem 3.29]{RSVarga_2000} gives us a bound on the spectral
radius of the stationary iteration whose error-propagation operator is
given by $I - \hat{M}^{-1}A$, as
\[
\rho\left(I-\hat{M}^{-1}A\right) =
\frac{\rho\left(A^{-1}\hat{N}\right)}{1+\rho\left(A^{-1}\hat{N}\right)}
< 1.
\]
We note that the quotient given is monotone increasing with
$\rho\left(A^{-1}\hat{N}\right)$, so that any upper bound that we get on
this spectral radius gives an upper bound on that of
$I-\hat{M}^{-1}A$,
\[
\rho\left(A^{-1}\hat{N}\right) \leq K \Rightarrow
\rho\left(I-\hat{M}^{-1}A\right) \leq \frac{K}{1+K}.
\]
A natural bound to use is that
\[
\rho\left(A^{-1}\hat{N}\right) \leq \left\|A^{-1}\hat{N}\right\| 
\leq \left\|A^{-1}\right\|\left\|\hat{N}\right\|,
\]
where we will consider the standard matrix norm induced by the
discrete maximum norm.
By a standard barrier-function technique \cite{Ro96}, there exists
a constant, $C$, such that $\left\|A^{-1}\right\|
\leq C$ (taking the vector, $W$, whose value at the degree of freedom
associated with grid point $(x_i,y_j)$ is $1+x_i$, so that $AW \geq
\Clow$ (or $\Clow_x$) in the pointwise sense, but $\|W\|\leq 2$).  For the bound on
$\|\hat{N}\|$, we note that $\hat{N}$ has at most two
nonzero entries in each of its rows or columns (by construction), and
that all of these entries are of the form of either
$\eps/(\overline{h}_ih_i)$ or $\eps/(\overline{k}_jk_j)$.  All of
the entries in $\hat{N}$, however, are associated with points
$(x_i,y_j)$ where there exists a
constant, $C$, such that $h_i, \overline{h}_i, k_j, \overline{k}_j
\geq C/N$ for all such entries dropped from $A$.  Thus, there exists a constant,
$C$, such that $\|\hat{N}\| < C\eps N^2$.  Taken together,
these give us the bound that
\[
\rho\left(I-\hat{M}^{-1}A\right) \leq \frac{C\eps N^2}{1+C\eps N^2}.
\]
We note that this bound is suboptimal in comparison to Theorem \ref{thm:big_bound}, since we typically assume that
$\eps N \ll 1$, but not that $\eps N^2$ is bounded by a constant.  Nonetheless, it is an
improvement on standard bounds on the condition number of the
unpreconditioned system, $\kappa(A) \leq C\frac{N^2}{\eps(\ln N)^2}$,
or a diagonally preconditioned system, $\kappa(\hat{D}^{-1}A) \leq
C\frac{N^2}{\ln N}$ \cite{Ro96}.  As always with nonsymmetric systems,
convergence of either a stationary or Krylov iteration depends on much
more than the condition number of the system; however, this is an
indication that the preconditioner construction should lead to
improved performance for iterations preconditioned in this way.

\subsection{Multigrid for the corner region}
\label{sec:multigrid in the corner}
Finally, we consider the case of the approximation of $A_{CC}$,
corresponding to mesh points in
$[0,\tau_x]\times[0,\tau_y]$. 
 In this region, the mesh is refined in both the $x$- and
$y$-directions, leading to discrete problems where the diffusion terms
in both directions are no longer dominated by the convection terms.
In such cases, multigrid methods are well-recognized as providing
excellent approximations to $A_{CC}^{-1}$ that can be implemented with
$\Oh(N^2)$ computational cost.  Here, we discuss the details of the
construction of such methods.  Since the methods we adopt are quite
different for the two cases we consider, we present the methods
independently.

On a classical Shishkin mesh for a problem with one parabolic and one
exponential boundary layer, the transition points are as
in~\eqref{eq:tau parabolic}, and
we expect $h_i = h_{i+1} \approx 2\eps\ln(N)/N$
and $k_j = k_{j+1} \approx 2\sqrt{\eps}\ln(N)/N$.  This gives off-diagonal
entries of size $\Oh(N^2/\ln^2(N))$ in the North and South directions on the
mesh, but of size $\Oh(N^2/(\eps\ln^2(N)))$ in the West and East directions.  In
essence, the problem much more resembles a classical anisotropic
diffusion operator than a singularly perturbed convection-diffusion
operator.  As a result, we approximate $A_{CC}^{-1}$ by the action of
a multigrid cycle appropriate to an anisotropic problem.  In
particular, we make use of a semi-coarsening multigrid algorithm in
this case, where the coarse grids are formed by factor-2 coarsening in
only the $x$-direction.  As a relaxation scheme, we use pointwise
Gauss-Seidel, again ordered in a ``downstream'' direction, ordered
from the top right point in the corner region to the bottom left.  We
use a standard V(1,1) cycling strategy, and approximate a solve on the
coarsest grid by four sweeps of the downstream Gauss-Seidel relaxation.

In order to properly account for possible variations in the mesh size
and the effects of the convection term, we use a Galerkin coarsening
algorithm, with each coarse-grid operator formed by the triple-product
of a restriction operator, the fine-grid operator, and an
interpolation operator.  For ease of construction, on the finest grid
(the discretization mesh), we perform a row-wise rescaling of the
finite-difference discretization (only within the multigrid cycle on
the corner region, with a corresponding rescaling on the residual in
this region before the cycle is applied), multiplying the row of the
matrix that corresponds to node $(x_i,y_j)$ by
$\overline{h}_i\overline{k}_j$; in essence, this rescales the problem
from a finite-difference-like scaling to one more akin to a
finite-element discretization, where Galerkin coarsening is more
natural.  With this rescaling, we define a one-dimensional
interpolation operator from the coarse grid to fine-grid node
$(x_i,y_j)$ by first ``collapsing'' the matrix stencil corresponding
to this row in the North-South direction into a 3-point operator.  Adopting the notation of writing
$a_{(i,j),(k,\ell)}$ for the entry in the matrix in row corresponding
to node $(x_i,y_j)$ and column corresponding to node $(x_k,y_\ell)$,
we define the interpolation operator to fine-grid node $(x_i,y_j)$
with entries
\begin{align*}
&-\frac{a_{(i,j),(i-1,j-1)}+a_{(i,j),(i-1,j)}+a_{(i,j),(i-1,j+1)}}{a_{(i,j),(i,j-1)}+a_{(i,j),(i,j)}+a_{(i,j),(i,j+1)}}\\
\text{and }&
-\frac{a_{(i,j),(i+1,j-1)}+a_{(i,j),(i+1,j)}+a_{(i,j),(i+1,j+1)}}{a_{(i,j),(i,j-1)}+a_{(i,j),(i,j)}+a_{(i,j),(i,j+1)}},
\end{align*}
for the weights of interpolation to fine-grid node $(x_i,y_j)$ from the coarse-grid nodes associated
with points $(x_{i-1},y_j)$ and $(x_{i+1},y_j)$, respectively.
Note that, while we have only a five-point stencil on the finest grid,
the use of such Galerkin coarsening leads to nine-point stencils on
all coarse grids, so we define the interpolation operator for the
general case, and use a similar formula (adapted only to account for
the coarsening) on all grids.  On all grids, we use the transpose of
this operator as the restriction operator.  Such an operator-induced
interpolation operator is inspired by the BoxMG algorithm
\cite{REAlcouffe_ABrandt_JEDendy_JWPainter_1981a,JEDendy_1982a}, which
uses a similar technique for anisotropic problems.

In preliminary numerical experiments, we found that using a single
cycle of the above scheme did not lead to scalable results for a
reasonable range of values for $N$ and $\eps$.  Instead, we use a
residual-reduction based tolerance, with $M_{CC}^{-1}$ defined by
performing as many cycles of the above method as needed to reduce the
residual over the corner region of the mesh by a relative factor of
$10^2$.  In results reported below, this requires only 3 V-cycles;
however, in other experiments, one or two more cycles were sometimes needed to reach this
tolerance.  Other options for gaining more robustness would be to
increase the number of relaxation sweeps used on each level, or to
switch to using a direct solver for the coarsest-grid system, but
neither of these were thoroughly explored, as the strategy above did
not lead to any apparent outliers in the data.

For the case of two exponential layers the
transition points are of the same order of magnitude, see
\eqref{eq:tau exponential}, and so the off-diagonal terms in
$A_{CC}$ are much more balanced in the $x$- and $y$-directions,
allowing a simpler cycling structure.  Here, we use a full-coarsening
multigrid algorithm, coarsening by a factor of two in each direction.
We use rediscretization to define the coarse-grid operators, and
define geometric (bilinear) interpolation that accounts for the mesh
spacing within the corner region.  On Shishkin meshes, this coincides
with the classical bilinear interpolation operator on uniform meshes,
with interpolation weights of $1/2$ for fine-grid points that are
directly adjacent to two coarse-grid points, and weights of $1/4$ for
fine-grid points that are the centre of a coarse-grid cell.  We again
rescale by factors of $\overline{h}_i\overline{k}_j$ from the
finite-difference to finite-element style of scaling, allowing us to
use the transpose of this interpolation operator as restriction.  (We
note that such rescaling can be avoided on uniform meshes if one uses
``full weighting'' restriction, but this is equivalent to what we do.)
The cycling structure in this case matches that for the case of one
parabolic and one exponential layer, with the same approximate
coarse-grid solve.  Here, we found better results by defining
$M_{CC}^{-1}$ to correspond to stationary iteration with this cycle as
needed to reduce the residual over the corner region of the mesh by a
relative factor of $10^3$, which is again quickly reached (in 5
iterations for the results reported below).

We note that the approaches described above both differ significantly
from the method of \cite{GaClLi02}.  That paper described a multigrid
algorithm to be applied to the same discretization on Shishkin meshes,
but aimed at preconditioning the full system, and not just the
discretization in the corner region.  There, full-coarsening multigrid
was applied using rediscretized coarse-grid operators and tuned
intergrid transfer operators that were derived from the Shishkin mesh
structure and the problem under consideration (with two exponential
layers).  Furthermore, an alternating line Gauss-Seidel relaxation was
used.  While the method proposed in \cite{GaClLi02} was generally
successful, our overall preconditioner has a lower cost per cycle,
because it focuses the numerical effort on the region of the mesh
where it is needed most.

\subsection{Numerical examples}\label{sec:2D numbers}
We test the preconditioner developed above to solve two model
problems, one that exhibits both a parabolic and an exponential layer,
and one that exhibits two exponential layers.  In both cases, we use
the preconditioner described above with FGMRES \cite{YSaad_2003a} as
the outer Krylov method.  As in one dimension (and as discussed
above), we have the bound that $\|A^{-1}\| < C$, for some constant $C$
that is independent of $\eps$ and $N$, so we use a direct
residual-based stopping tolerance on the expected almost-first-order
discretization error on the Shishkin meshes considered here, iterating
until $\|R^{(k)}\| \leq 10N^{-1}\ln(N)$.  While the above bound is on
the discrete maximum norm of the matrix, the nature of FGMRES requires
the stopping tolerance to be evaluated in the Euclidean norm, which we
do. The algorithm is implemented in C and was compiled using
gcc (version 8.2.0).
All numerical results in this section were run, \fix{in serial, on a
             single core of a 2.4 GHz Xeon processor on a system with 
192 GiB of RAM}.  For comparison, we consider a
direct solution of the same linear systems using UMFPACK~\cite{Da04}.

As a first example, for the case of one parabolic layer and one
exponential layer, we consider the solution of
  \[
-\eps \Delta u - u_x + u = f \text{ on }(0,1)^2,
\]
with $f(x,y)$ chosen to yield a manufactured solution of
\begin{equation}
  \label{eq:parabolic_MMS}
u(x,y) = \left(\cos\left(\frac{\pi x}{2}\right)-
\frac{e^{-x/\eps}-e^{-1/\eps}}{1-e^{-1/\eps}}\right)\left(\frac{1 - e^{-y/\sqrt{\eps}}}{ 1 - e^{-1/\sqrt{\eps}}}-y^{5/2}\right).
\end{equation}

To validate both the discretization and the chosen stopping tolerance,
Table \ref{tab:parabolic_errors} shows discretization errors for the
discrete solutions found by the algorithm.  We note that these show
the expected steadiness as $\eps \rightarrow 0$, and the expected
decay with large $N$.  Preconditioned FGMRES iteration counts are
shown in parentheses in Table \ref{tab:parabolic_times_iters}.  Here, we see that the
iteration counts are quite steady as $\eps \rightarrow 0$ and for
varying values of $N$, aside from in the top-right corner of the
table.  Here, $\eps N$ is not small enough for our
heuristics to suggest that we are in the right range for the
preconditioner to be effective, so the degradation in performance is
not too surprising.

\begin{table}[htb]
  \begin{center}
  \caption{Approximation error, measured in the discrete maximum norm,
    for manufactured solution in \eqref{eq:parabolic_MMS} generated by
    preconditioned FGMRES.}
  \label{tab:parabolic_errors}
\setlength{\tabcolsep}{3pt}
\begin{tabular}{c||c|c|c|c|c}
  \hline
  $\eps$   & $N = 2^7$ &  $N = 2^8$ &  $N = 2^9$ &  $N = 2^{10}$ &  $N = 2^{11}$ \\
  \hline
  $10^{-5}$ & $3.822\times 10^{-2}$ & $2.204\times 10^{-2}$ & $1.242\times 10^{-2}$ & $6.915\times 10^{-3}$ & $3.783\times 10^{-3}$  \\
  $10^{-6}$ & $3.823\times 10^{-2}$ & $2.205\times 10^{-2}$ & $1.244\times 10^{-2}$ & $6.903\times 10^{-3}$ & $3.783\times 10^{-3}$  \\
  $10^{-7}$ & $3.823\times 10^{-2}$ & $2.205\times 10^{-2}$ & $1.244\times 10^{-2}$ & $6.902\times 10^{-3}$ & $3.783\times 10^{-3}$  \\
  $10^{-8}$ & $3.823\times 10^{-2}$ & $2.205\times 10^{-2}$ &
                                                              $1.244\times 10^{-2}$ & $6.902\times 10^{-3}$ & $3.783\times 10^{-3}$ \\
  \hline
  \end{tabular}
  \end{center}
\end{table}

\begin{table}[htb]
  \begin{center}
\caption{CPU times (in seconds) for preconditioned FGMRES to achieve residual
  stopping tolerance for manufactured solution in
  \eqref{eq:parabolic_MMS}.  Corresponding iteration counts are given
  in parentheses.}
\label{tab:parabolic_times_iters}
  \begin{tabular}{c||c|c|c|c|c}
    \hline
  $\eps$   & $N = 2^7$ &  $N = 2^8$ &  $N = 2^9$ &  $N = 2^{10}$ &  $N = 2^{11}$ \\
      \hline
$10^{-5}$ & 0.007 (3) & 0.037 (4) & 0.205 (5) & 1.536 (9) & 14.804 (23) \\
$10^{-6}$ & 0.007 (3) & 0.028 (3) & 0.165 (4) & 0.868 (5) & 4.581 (8) \\
$10^{-7}$ & 0.007 (3) & 0.036 (4) & 0.165 (4) & 0.707 (4) & 3.525 (5) \\
$10^{-8}$ & 0.009 (4) & 0.036 (4) & 0.165 (4) & 0.868 (5) & 3.524 (5) \\
      \hline
    \end{tabular}
\end{center}
\end{table}

The timing data presented in Table \ref{tab:parabolic_times_iters}
shows that the CPU times scale largely as expected, growing
proportionately to iteration counts and problem sizes.  We note that
for $\eps < 10^{-5}$, the solution time for each problem on the
$1024\times 1024$ mesh is less than 1 second.  This includes both the
setup of the preconditioners (assembling the tridiagonal systems for
the two edge-layer regions, performing the forward sweep of the Thomas
algorithm to factor these systems, and computing all necessary
components on all levels of the multigrid algorithm for the corner
region) and the preconditioned FGMRES solve time (including the
residual and preconditioned residual calculations, modified
Gram-Schmidt and Arnoldi steps, and construction of the solution once
converged).  As a comparison, Table \ref{tab:parabolic_speedup} shows
the speedup factors achieved for the preconditioned FGMRES iteration
over a direct solution using \fix{UMFPACK} on the same machine.  As
expected, the direct solution cost grows faster than $\Oh(N^2)$ as the
mesh is refined and, so, the speedup generally increases with larger
$N$ (aside from the top-right corner, where iteration counts increase
for the preconditioned FGMRES iteration). \fix{Table \ref{tab:parabolic_speedup} also shows the \textit{number of digits of accuracy} in the preconditioned FGMRES solution, defined as $\log_{10}\left(\|U^E\|_\infty/\|U^E-U^{(k)}\|_\infty\right)$, where $U^E$ is the solution returned by UMFPACK (treated as the exact solution to the linear system) and $U^{(k)}$ is the iterative solution at iteration $k$, when the stopping tolerance is satisfied (as reported in Table \ref{tab:parabolic_times_iters}).  We note that for $\eps < 10^{-5}$, we consistently match the direct solution to at least 3 digits of accuracy, and often more.}

\begin{table}[htb]
  \begin{center}
\caption{Speedup of preconditioned FGMRES to achieve residual
  stopping tolerance for manufactured solution in
  \eqref{eq:parabolic_MMS} over direct solution using \fix{UMFPACK}.  \fix{In parentheses, the number of digits of
  accuracy in the preconditioned FGMRES solution.}}
\label{tab:parabolic_speedup}
  \begin{tabular}{c||c|c|c|c|c}
    \hline
  $\eps$   & $N = 2^7$ &  $N = 2^8$ &  $N = 2^9$ &  $N = 2^{10}$ &  $N = 2^{11}$ \\
      \hline
$10^{-5}$ &   11.32 (4.0) &   16.66 (3.6) &   26.73 (2.1) &   23.58 (1.7) &   17.00 (2.6) \\
$10^{-6}$ &   13.28 (5.7) &   22.66 (4.3) &   29.35 (4.4) &   34.82 (3.4) &   43.68 (3.1) \\
$10^{-7}$ &   13.80 (4.1) &   17.96 (5.7) &   29.42 (5.0) &   39.16 (4.5) &   47.29 (5.1) \\
$10^{-8}$ &   10.79 (5.9) &   16.82 (4.6) &   25.32 (4.1) &   27.38 (4.3) &   40.68 (5.2) \\
      \hline
    \end{tabular}
\end{center}
\end{table}

Our second example is for the case of two exponential layers, where we
solve
\[
-\eps \Delta u - 2u_x -3u_y + u = f \text{ on }(0,1)^2,
\]
with $f(x,y)$ chosen to yield a manufactured solution of
\begin{equation}
  \label{eq:exponential_MMS}
  u(x,y) = \left(\cos\left(\frac{\pi x}{2}\right)\right)\left(1-e^{-2x/\eps}\right)\left(1-y\right)^3\left(1-e^{-3y/\eps}\right)
\end{equation}

Again, to validate both the discretization and the chosen stopping
tolerance, Table \ref{tab:exponential_errors} presents discretization
errors for the solutions found by the algorithm, for a slightly
different range of values of $\eps$ because of the different layer
structure in this case.  As before, this shows performance that is
clearly bounded independently of $\eps$ and decays in the expected way
with increasing $N$.  Preconditioned FGMRES iterations are shown in
parentheses in Table \ref{tab:exponential_times_iters}, and are very
similar to those seen for the case of one parabolic and one
exponential layer.  As before, we note the degradation in performance
in the upper-right corner of the table, for larger values of $\eps N$
where the heuristics motivated above do not apply.

\begin{table}[htb]
  \begin{center}
  \caption{Approximation error, measured in the discrete maximum norm,
    for manufactured solution in \eqref{eq:exponential_MMS} generated by
    preconditioned FGMRES.}
  \label{tab:exponential_errors}
\setlength{\tabcolsep}{3pt}
  \begin{tabular}{c||c|c|c|c|c}
    \hline
  $\eps$   & $N = 2^7$ &  $N = 2^8$ &  $N = 2^9$ &  $N = 2^{10}$ &  $N = 2^{11}$ \\
      \hline
$10^{-4}$ & $3.728\times 10^{-2}$ & $2.260\times 10^{-2}$ & $1.323\times 10^{-2}$ & $7.570\times 10^{-3}$ & $4.248\times 10^{-3}$  \\
$10^{-5}$ & $3.729\times 10^{-2}$ & $2.261\times 10^{-2}$ & $1.325\times 10^{-2}$ & $7.572\times 10^{-3}$ & $4.248\times 10^{-3}$  \\
$10^{-6}$ & $3.729\times 10^{-2}$ & $2.261\times 10^{-2}$ & $1.325\times 10^{-2}$ & $7.572\times 10^{-3}$ & $4.248\times 10^{-3}$  \\
$10^{-7}$ & $3.730\times 10^{-2}$ & $2.261\times 10^{-2}$ & $1.325\times 10^{-2}$ & $7.572\times 10^{-3}$ & $4.248\times 10^{-3}$  \\
      \hline
  \end{tabular}
  \end{center}
\end{table}

\begin{table}[htb]
  \begin{center}
\caption{CPU times (in seconds) for \fix{preconditioned} FGMRES to achieve residual
  stopping tolerance for manufactured solution in
  \eqref{eq:exponential_MMS}.  Corresponding iteration counts are given
  in parentheses.}
\label{tab:exponential_times_iters}
  \begin{tabular}{c||c|c|c|c|c}
    \hline
  $\eps$   & $N = 2^7$ &  $N = 2^8$ &  $N = 2^9$ &  $N = 2^{10}$ &  $N = 2^{11}$ \\
      \hline
$10^{-4}$ & 0.010 (3) & 0.041 (4) & 0.266 (6) & 2.656 (14) & 37.690 (40) \\
$10^{-5}$ & 0.010 (4) & 0.041 (4) & 0.178 (4) & 1.103 (6)  & 7.625  (10) \\
$10^{-6}$ & 0.010 (4) & 0.041 (4) & 0.221 (5) & 0.921 (5)  & 3.771  (5) \\
$10^{-7}$ & 0.010 (4) & 0.051 (5) & 0.221 (5) & 0.923 (5)  & 4.519  (6) \\
      \hline
    \end{tabular}
\end{center}
\end{table}

Considering the solution times shown in Table \ref{tab:exponential_times_iters}, we again see scaling as expected, with total time-to-solution that scales directly with iteration counts and problem sizes.  Notably, the $\Oh(N^2)$ cost per iteration clearly scales through the largest problem size, and loss of scalability in solve times for large $N$ is directly due to increasing iteration counts in the case where $\eps N$ is large.  As before, the reported timings include all costs for the setup and preconditioned FGMRES iterations, and we note that the cost per iteration is quite comparable for the preconditioner in this case to that of one parabolic and one exponential layer.  Again, we see substantial speedups over a direct solve using \fix{UMFPACK}, documented in Table \ref{tab:exponenetial_speedup}, contrasting the poor scaling in $N$ of a direct solver with the $\Oh(N^2)$ total solution cost seen here for smaller values of $\eps$.
\fix{Table \ref{tab:exponenetial_speedup} also shows the number of digits of accuracy in the preconditioned FGMRES solution, calculated as above.  As before, we see that we achieve at least 2 digits of accuracy in all cases, and generally more for smaller $\eps$.}

\begin{table}[htb]
  \begin{center}
\caption{Speedup of preconditioned FGMRES to achieve residual
  stopping tolerance for manufactured solution in
  \eqref{eq:exponential_MMS} over direct solution using
  \fix{UMFPACK}.  \fix{In parentheses, the number of digits of
  accuracy in the preconditioned FGMRES solution.}}
\label{tab:exponenetial_speedup}
  \begin{tabular}{c||c|c|c|c|c}
    \hline
  $\eps$   & $N = 2^7$ &  $N = 2^8$ &  $N = 2^9$ &  $N = 2^{10}$ &  $N = 2^{11}$ \\
      \hline
$10^{-4}$ &   6.81 (3.3) &   13.92 (3.0) &   18.73 (2.6) &   13.03 (3.3) &    7.05 (3.6)  \\ 
$10^{-5}$ &   7.75 (5.8) &   13.92 (5.5) &   27.84 (4.2) &   29.68 (4.3) &   26.69 (4.0)  \\ 
$10^{-6}$ &   8.04 (4.8) &   13.72 (4.4) &   19.85 (7.0) &   29.02 (6.9) &   43.09 (3.8)  \\ 
$10^{-7}$ &   8.33 (3.8) &   12.04 (5.9) &   18.94 (6.0) &   27.64 (5.6) &   32.51 (4.6)  \\ 
      \hline
    \end{tabular}
\end{center}
\end{table}

\section{Conclusions}
\label{sec:conclusions}
In this paper, we have extended the ideas of boundary-layer
preconditioning for singularly perturbed problems, first proposed for
the reaction-diffusion case in \cite{SMacLachlan_NMadden_2013a}, to
the case of convection-diffusion.  As is typically the case, the
extension from the symmetric to non-symmetric case requires the
development of new tools to extend the theory accordingly, but we are
able to provide a sharp bound on the conditioning of the
preconditioned system in one dimension, and a weaker bound in two
dimensions.  Numerical results demonstrate excellent performance of
the preconditioner in one and two dimensions, for problems with both
exponential and parabolic layers.

In future work, we will consider the extension of these
preconditioners to finite-element discretizations of both linear
convection-diffusion problems and nonlinear problems with boundary
layers, such as Navier-Stokes flow in a channel.  We note that while
the work presented here focuses on the case of boundary layers, there
is no conceptual restriction that prevents applying the technique to
interior layers, so long as the layer structure in the mesh is
available for construction of the preconditioner.  \fix{In the case of non-regular domains or unstructured grids, the regions defined above can be generalized based on whether they include refinement in zero, one, or two dimensions.  In regions of no refinement (corresponding to the interior region above), knowledge of node location and convection coefficients can be used to develop a downstream ordering for the Gauss-Seidel approximation.  In regions of refinement in one dimension, line relaxation can be generalized based on bin-sorting of geometric coordinates along the non-refined direction.  Finally, in regions with refinement in both directions, algebraic multigrid can be used to replace the geometric multigrid used here.  While this clearly requires more information than is typically used in global algebraic multigrid approaches, it is feasible given basic information about the geometry, mesh, and coefficients in the PDE.}

\fix{Another possible direction for future work would be the extension
                of these techniques to three-dimensional problems.
                Here, as in the reaction-diffusion case discussed in
                \cite[\S 5]{SMacLachlan_NMadden_2013a}, the number of types of regions in the mesh increases, but no fundamental changes occur in the strategy.  Depending on the number of nonzero values in the convection coefficient, a region of the mesh may be of finer resolution in zero, one, two, or three spatial dimensions.  The cases of refinement in zero or one dimension are similar to those discussed here, while appropriate plane solves (using multigrid methods appropriate for two dimensions) would be needed for regions with two refined dimensions, and a fully coupled solve (using multigrid methods appropriate for three dimensions) would be needed in any corner regions with three refined dimensions.}

\bibliographystyle{siam}
\bibliography{spp}

\end{document}